\documentclass[11pt,reqno]{amsart}
\usepackage{amssymb}
\usepackage{graphicx}
\usepackage{caption}
\usepackage{url}

{\theoremstyle{plain}
 \newtheorem{theorem}{Theorem}
 \newtheorem{corollary}{Corollary}

\newtheorem{lemma}{Lemma}
}

\begin{document}

\title[Uniform distribution of $\alpha n$ modulo one]{Uniform distribution of $\alpha n$ modulo one \\ for a family of integer sequences}

\begin{abstract} 
We show that the sequence $(\alpha n)_{n\in \mathcal{B}}$ is uniformly 
distributed modulo 1, for every irrational $\alpha$, provided $\mathcal{B}$ belongs 
to a certain family of integer sequences, which includes the prime, almost prime, 
squarefree, practical, densely divisible and lexicographical numbers.
We also give an estimate for the discrepancy if $\alpha$ has finite irrationality measure. 
\end{abstract}

\author{Andreas Weingartner}
\address{ 
Department of Mathematics,
351 West University Boulevard,
 Southern Utah University,
Cedar City, Utah 84720, USA}
\email{weingartner@suu.edu}
\maketitle

\section{Introduction}

We say that a sequence $(a_n)_{n\in \mathbb{N}} $ of real numbers is uniformly distributed modulo $1$ (u.d. mod 1) if
$$
\lim_{x\to \infty} \frac{|\{n\le x: \{a_n\} \in [a,b] \}|}{x} = b-a
$$ 
for all $0\le a<b\le 1$, where $\{u\}$ denotes the fractional part of $u$.  
Weyl's criterion asserts that this is equivalent to 
$$
\lim_{x\to \infty}  \frac{1}{x} \sum_{n\le x} e(l a_n) = 0,
$$
for every fixed non-zero integer $l$, where $e(y):=e^{2\pi i y}$. 

Let $\alpha$ be any irrational real number.
Weyl's criterion shows at once that the sequence $(\alpha n)_{n\in \mathbb{N}}$ is u.d. mod 1. 
Vinogradov \cite[Ch. XI]{Vino} proved that the same holds for the sequence $(\alpha p)_{p\in \mathbb{P}}$, where $p$ runs through the prime numbers.

We consider the following family of integer sequences. For each natural number $n$, let $I(n)$ be an arbitrary interval of real 
numbers, possibly empty, or the union of a bounded number (uniformly in $n$) of such intervals. 
Let $\mathcal{B}=\mathcal{B}_I $ be the set of positive integers containing $n=1$ and all those $n \ge 2$ with prime factorization $n=p_1 \cdots p_k$, $p_1\le p_2\le \ldots \le p_k$, which satisfy 
\begin{equation*}\label{Bdef}
p_{j} \in I\big( \prod_{1\le i<j} p_i \big) \qquad (1\le j \le k).
\end{equation*}

This setting is more general than in \cite{SPA}, where $I$ was of the form $I(n)=[1,\theta(n)]$ for some function $\theta(n)$. 
If $I(1)=[1,\infty)$ and $I(n)=\emptyset$ for $n>1$, then $\mathcal{B}\setminus \{1\}$ is the set of primes. 
If $I(1)=I(p)=[1,\infty)$ for primes $p$ and $I(n)=\emptyset$ for composite $n$, then $\mathcal{B}$ is the set of integers with at most two prime factors,
counted with multiplicity. Similarly, one can obtain more general almost-primes.
Squarefree numbers can be generated with $I(n)=(P(n),\infty)$, where $P(n)$ denotes the largest prime factor of $n>1$ and $P(1)=1$. 
Integers whose divisors grow by factors of at most $t$, which are called \emph{$t$-densely divisible} \cite{Poly} or \emph{$t$-dense} \cite{Tau},
arise from $I(n)=[1,nt]$, while the \emph{practical numbers} result from $I(n)=[1,\sigma(n)+1]$, where $\sigma(n)$ is the sum of the positive divisors of $n$. 
This family also includes the \emph{lexicographical numbers} \cite{ST}, where $I(n)=\{P(n)\}\cup (n,\infty)$.
Define
$$
\mathcal{B}(x) := \mathcal{B}\cap [1,x], \quad B(x):=|\mathcal{B}(x)|, \quad B_d(x):=|\{n \in \mathcal{B}(x) : d|n\}|.
$$

\begin{theorem}\label{thm1}
Let $\alpha$ be any irrational real number.
Assume there exist constants $A\ge 0$ and $\delta >0$ such that
\begin{equation}\label{Bcond}
B(x)\gg x (\log x)^{-A}, \qquad B_d(x)\ll  d^{-\delta}B(x), \qquad (x,d \ge 2).
\end{equation}
Then the sequence $(\alpha n)_{n \in \mathcal{B}}$ is u.d. mod 1. 
\end{theorem}

When $\alpha>1$, the integer sequence $(\lfloor \alpha n \rfloor)_{n\in \mathbb{N}}$ contains 
a proportion of $\frac{1}{\alpha}$ of the natural numbers.
Corollary \ref{cor0}, which follows easily from Theorem \ref{thm1} (see Sec. \ref{cor0proof}), 
shows that it contains the same proportion of members of $\mathcal{B}$.
\begin{corollary}\label{cor0}
Suppose that the assumptions of Theorem \ref{thm1} are satisfied and $\alpha >1$.
Then
\begin{equation}\label{eqcor0}
|\{n \in \mathbb{N}: \lfloor \alpha n \rfloor  \in \mathcal{B}(x) \}| \sim \frac{1}{\alpha} B(x) \quad (x \to \infty).
\end{equation}
\end{corollary}

If $I(n)=[1,\theta(n)]$, the following assumption implies \eqref{Bcond},
by Lemma \ref{lemBmult}:
There exist constants $ C,J, K\ge 1$ and  $0\le \eta <1$, 
such that for all $m, n \in \mathbb{N}$, 
\begin{equation}\label{thetahyp}
\theta (n) \le \theta(mn)\le Cm^J \theta(n) , \quad \max(2,n)\le \theta(n) \le K n \exp((\log n)^\eta).
\end{equation}

\begin{corollary}\label{cor1}
Let $\alpha$ be irrational. Assume $I(n)=[1,\theta(n)]$ and \eqref{thetahyp} holds.
Then the sequence $(\alpha n)_{n \in \mathcal{B}}$ is u.d. mod 1. If $\alpha>1$, then \eqref{eqcor0} holds.
\end{corollary}

\renewcommand{\arraystretch}{1.2}
\begin{table}[h]
     \begin{tabular}{| c | c |  c |  }
    \hline
  Sequence $\mathcal{B}$ &  $ I(n) $ & OEIS \cite{oeis} \\ \hline \hline
almost prime & see above & A037143 \\ \hline
squarefree  & $(P(n),\infty)$ & A005117 \\ \hline
lexicographical & $\{P(n)\} \cup (n,\infty)$ & A361232 \\ \hline
    $t$-dense  ($t\ge 2$) & $[1,nt]$ &   A174973 ($t=2$)\\ \hline 
 practical & $[1,\sigma(n)+1]$  & A005153  \\ \hline 
 practical \& $\varphi$-practical & $[1,n+1]$ & A359420 \\ \hline 
   Nullwertzahlen & $[1,\max(2,n)]$ &  A047836 \\ \hline 
\end{tabular}    
 \caption{Examples of $\mathcal{B}$ with $(\alpha n)_{n\in \mathcal{B}}$ u.d. mod 1.
It's easy to verify \eqref{Bcond} for the first two examples, and for the third with \cite[Thm. 2]{ST}. For the others it follows from 
Corollary \ref{cor1}.}\label{table1}
\end{table}
\renewcommand{\arraystretch}{1}

\begin{corollary}
Let $\alpha$ be irrational.
If $\mathcal{B}$
is one of the sequences in Table \ref{table1}, then the sequence $(\alpha n)_{n \in \mathcal{B}}$ is u.d. mod 1.
If $\alpha>1$, then \eqref{eqcor0} holds.
\end{corollary}

The proof of Theorem \ref{thm1} 
is divided into two cases. If $\alpha$ is very close to a rational number $a/q$ with a
small denominator $q$ (relative to $x$), we say $\alpha$ belongs to the \emph{major arcs}.
If not, we say $\alpha$ belongs to the \emph{minor arcs}.  

\begin{theorem}[Major Arcs]\label{thm2}
Let $g(x)=\log x \log \log x$. There exists a constant $c>0$ such that the following holds. 
Let $U>0$ be fixed. 
For  $q \le (\log x)^U$, $\alpha = a/q+\beta$, where $(a,q)=1$ and $|\beta| \le \exp(c\sqrt[3]{g(x)}) /x$, we have
\begin{equation}\label{eq1thm2}
\sum_{n\in \mathcal{B}(x)} e(\alpha n)  =  \sum_{n \in \mathcal{B}(x)}  \frac{\mu(q/(n,q))}{\varphi(q/(n,q))} e(\beta n) + O( x \exp(-2c \sqrt[3]{g(x)})),
\end{equation}
where $(n,q)=\gcd(n,q)$, $\mu$ is the M\"obius function and $\varphi$ is Euler's totient function.
If, in addition, $B(x) \gg x \exp(-c\sqrt[3]{g(x)})$ and
$B_d(x) \ll d^{-\delta}B(x)$ for $x,d \ge 2$ and some constants $0<\varepsilon <\delta \le 1$, then
$$
\sum_{n\in \mathcal{B}(x)} e(\alpha n) \ll \frac{B(x)}{ q^{\delta-\varepsilon}} .
$$
\end{theorem}

\begin{theorem}[Minor Arcs]\label{thm3}
Let $\kappa<1/\sqrt{6}$ be a constant. Let $L=\log x$ and 
$$1 \le R \le \exp(\kappa \sqrt{\log x \log\log x}).$$ 
Then, for $h,a,q\in \mathbb{N}$, $h\le R$,
$R^{12} L^{26}<q\le \frac{x}{R^{11}L^{26}}$ and $|\alpha -a/q|\le 1/q^2$, where $(a,q)=1$, we have
$$
\sum_{n\in \mathcal{B}(x)} e(\alpha h n) \ll \frac{x }{R} .
$$
\end{theorem}

Theorem \ref{thm3} leads to an estimate for the discrepancy of the sequence $(\{\alpha n\})_{n\in \mathcal{B}}$,
assuming the denominators of the continued fraction convergents $a_j/q_j$ of $\alpha$ don't grow too quickly.  
The irrationality measure of an irrational number $\alpha$ can be defined as 
$\mu(\alpha)=1+ \limsup_{j\ge 1} \frac{\log q_{j+1}}{\log q_j}$.
Every algebraic irrational number $\alpha$ satisfies $\mu(\alpha)=2$, 
as well as $\mu(e)=2$, while the constants $\pi, \pi^2,  \log 2,  \log 3$ and $\zeta(3)$ are all known to have finite irrationality measure \cite{Weis}.
For $0\le y \le 1$, define
$$
B(x,y,\alpha) := |\{n\in \mathcal{B}(x) : \{\alpha n\} \le y \}|.
$$

The following estimate for the discrepancy follows from Theorem \ref{thm3}
and the  Erd\H{o}s-Tur\'an inequality (see Lemma \ref{lemET}).
\begin{theorem}\label{thm4}
Let $\kappa<\frac{1}{\sqrt{6}}$ and assume $\alpha$ is irrational with finite irrationality measure. Then, for all $x \ge x_0(\alpha)$,
$$
\sup_{0\le y \le 1} |B(x,y,\alpha) - y B(x)| \ll \frac{x}{\exp(\kappa\sqrt{\log x \log\log x})} .
$$
In particular, the sequence $(\alpha n)_{n \in \mathcal{B}}$ is u.d. mod 1
provided $B(x)$ satisfies $B(x) \gg x \exp(-\kappa \sqrt{\log x \log\log x})$ for some $\kappa < \frac{1}{\sqrt{6}}$.
\end{theorem}

\section{Derivation of Theorem \ref{thm1} from Theorems \ref{thm2} and \ref{thm3} }

By Weyl's criterion, we need to show that 
$$
T:=\frac{1}{B(x)} \sum_{n\in \mathcal{B}(x)} e(l \alpha n) \to 0 \quad (x\to \infty)
$$
for every fixed non-zero integer $l$. It suffices to consider $l=1$, since the following is equally
valid if $\alpha$ is replaced by $l \alpha$.
Let $(a_j/q_j)_{j\ge 1}$ be the sequence of continued fraction convergents of $\alpha$. 
Then $q_1<q_2<\ldots$ and 
$$
|\alpha - a_j/q_j | \le \frac{1}{q_j q_{j+1}} \quad (j\ge 1).
$$
Let $x$ be sufficiently large and let $R=(\log x)^{A+1}$ and $h=1$ in Theorem \ref{thm3}.
Define $Q=R^{12} L^{26}=L^{12A+38}$, where $L=\log x$.
If there is at least one $j$ such that $Q < q_j \le x/Q$,
Theorem \ref{thm3} (with $q=q_j$) shows that we have $T \ll x (\log x)^{-A-1} /B(x) \ll 1/\log x$. 

If there does not exist a $j$ such that $Q<q_j\le x/Q$, let $i$ be such that 
$q_i \le Q$ and $q_{i+1}> x/Q$. We have $|\alpha - a_i/q_i|\le 1/(q_i q_{i+1}) \le 1/q_{i+1} < Q/x$. 
Theorem \ref{thm2}, with $U=12A+38$ and $q=q_i$, shows that $T\ll q_i^{-\delta/2} $. 
As $x\to \infty$, $i\to \infty$ and $q_i\to \infty$. 
Thus $T\to 0$ as $x\to \infty$.

\section{Derivation of Corollary \ref{cor0} from Theorem \ref{thm1}}\label{cor0proof}

Let $\lambda=\frac{1}{\alpha}<1$. Since $(\lambda m)_{m\in \mathcal{B}}$ is u.d. mod 1, we have
\begin{equation*}
\begin{split}
\frac{1}{\alpha} B(x) = \lambda B(x) & \sim |\{m \in \mathcal{B}(x): 1-\lambda < \{\lambda m\} <1 \}| \\
& = |\{m \in \mathcal{B}(x): \exists n \in \mathbb{N},\lambda m< n < \lambda m +\lambda \}| \\
& = |\{m \in \mathcal{B}(x): \exists n \in \mathbb{N},  m=\lfloor \alpha n \rfloor \}| \\
& = |\{n \in \mathbb{N}: \lfloor \alpha n \rfloor \in \mathcal{B}(x)\}| .
\end{split}
\end{equation*}

\section{Derivation of Corollary \ref{cor1} from Theorem \ref{thm1}}

\begin{lemma}\label{lemBmult}
Assume $\theta$ satisfies \eqref{thetahyp} and $I(n)=[1,\theta(n)]$.
For $x, d\ge 2$, we have
$$
B(x) = \frac{c_\theta x}{\log x}\left( 1+O\left(\frac{1}{(\log x)^{1-\eta}}\right)\right), \quad B_d(x) \ll B(x) \frac{ \log 2d}{d } ,
$$
where $c_\theta>0$. 
\end{lemma}

\begin{proof}
The first claim follows from \cite[Thm. 4]{Tau}. 
Lemma 8 of \cite{SW} shows that
$\{m: md \in \mathcal{B}_\theta\} \subset \mathcal{B}_{\theta_d}$
where $\theta_d(n) \le \theta(dn)$ for all $n\ge1$. 
Since $\theta(dn) \le C d^J \theta(n)$ for some constants $C, J$, by \eqref{thetahyp},
$$
\{m \le x/d : md \in \mathcal{B}_\theta\} \subset \{ m\le x/d : m \in \mathcal{B}_{Cd^J\theta} \}.
$$ 
The result now follows from \cite[Prop. 1]{PW}.
\end{proof}

\section{Lemmas for the major arcs}

Let $P(n)$ denote the largest prime factor of $n\ge 2$ and put $P(1)=1$. 
The following estimate follows from \cite[Eqs. (1.3), (1.4), (1.5)]{deBru}. 

\begin{lemma}[de Bruijn]\label{lemPsi}
For $y \ge  (\log x)^2$ and $u=\frac{\log x}{ \log y}\ge 1$, we have
$$
\Psi(x,y):=\sum_{n\le x \atop P(n) \le y} 1 \ll x \exp(-u\log u).
$$
\end{lemma}
A key ingredient for the major arcs is the Siegel-Walfisz theorem (see \cite[Thm. 6.9 and Cor. 11.21]{MV}). 
\begin{lemma}[Siegel-Walfisz]\label{lemSW}
There exists a constant $c_1>0$ such that the following holds. Let $A>0$. For $q\le (\log x)^A$ and $(a,q)=1$,
$$
\pi(x,a,q):= \sum_{p\le x \atop p\equiv a \bmod q} 1 = \frac{\pi(x)}{\varphi(q)} +O_A(x e^{-c_1 \sqrt{\log x}}),
$$
where $\pi(x):=\pi(x,0,1)$ and $\varphi$ is Euler's totient function.
\end{lemma}

\begin{lemma}[Ramanujan's Sum]\label{lembsum}
Let $a,q,d \in \mathbb{N}$ with $(a,q)=1$ and $d|q$. Then
$$
\sum_{n=1 \atop (n,q)=d}^q e(n a/q) = \mu(q/d),
$$
where $\mu$ is the M\"obius function.
\end{lemma}
\begin{proof}
Writing $n'=n/d$ and $q'=q/d$ we have $(a,q')=1$ and 
$$
\sum_{n=1 \atop (n,q)=d}^q e(n a/q) = \sum_{n'=1 \atop (n',q')=1}^{q'} e(n' a/q') = \mu(q')=\mu(q/d).
$$
The last sum is called Ramanujan's sum and is evaluated in \cite[Thm. 4.1]{MV}.
\end{proof}

\section{Proof of Theorem \ref{thm2}}
If $n \in \mathcal{B}$, $n>1$, we write $n=mp$ where $p=P(n)$ and $m\in \mathcal{B}$. 
We have
$$
f(\alpha) := \sum_{n\in \mathcal{B}(x)} e(\alpha n) 
=e(\alpha)+ \sum_{m\in \mathcal{B}(x)} \sum_{p\in J(x,m)} e(\alpha p m)
$$
where $J(x,m)$ is the (possibly empty) union of a bounded number of intervals 
$$J(x,m):=\{y\in \mathbb{R}: P(m)\le y \le x/m\} \cap I(m).$$
Define $g(x):=\log x \log\log x$. 
If $E_0$ denotes the contribution to $f(\alpha)$ from primes $p\le \exp(g(x)^{2/3})$, then 
$$
|E_0| \le \Psi(x, \exp(g(x)^{2/3}) )\ll x \exp(-g(x)^{1/3}/4),
$$ 
by Lemma \ref{lemPsi}. 
Thus we may replace $J(x,m)$ by
$$
\tilde{J}(x,m):=J(x,m) \cap (\exp(g(x)^{2/3}) , \infty).
$$
Note that if $mp\equiv b \bmod q$ where $(b,q)=d$ and the prime $p$ satisfies $p>q$, then $(m,q)=d$.
In this case we write $q'=q/d$, $b'=b/d$ and $m'=m/d$.  
For $r$ with $(r,q)=1$, let $\bar{r}$ be such that $\bar{r} r \equiv 1 \bmod q$. 

We have
\begin{equation*}
\begin{split}
f(a/q+\beta) & =E_0+ \sum_{d|q} \sum_{b=1 \atop (b,q)=d}^q 
\sum_{m\in \mathcal{B}(x) \atop (m,q)= d}
\sum_{p \in \tilde{J}(x,m) \atop mp\equiv b \bmod q} 
 e(({\textstyle \frac{a}{q}}+\beta) p m) \\
&  =E_0+ \sum_{d|q} \sum_{b=1 \atop (b,q)=d}^q e({\textstyle \frac{a}{q}} \, b) 
\sum_{m\in \mathcal{B}(x) \atop (m,q)= d}
\sum_{p\in \tilde{J}(x,m) \atop p\equiv b' \overline{m'} \bmod q'} 
 e(\beta p m) .
\end{split}
\end{equation*}
Since $p\ge \exp(g(x)^{2/3})$, we have $q \le (\log x)^U \ll (\log p)^{3U/2}$.
By the Siegel-Walfisz theorem, in the form of Lemma \ref{lemSW}, the sum over $p$ is
$$
\int_{\rho \in \tilde{J}(x,m) } e(\beta \rho m) d(\pi(\rho, b' \overline{m'} ,q')) 
= \int_{\rho \in \tilde{J}(x,m)} e(\beta \rho m) d(\pi (\rho)/\varphi(q') + E_1),
$$
where $E_1 \ll \rho e^{-c_1\sqrt{\log \rho}}$.  Let $0<c<c_1/3$. 
Integration by parts applied to the error term shows that 
$$
\sum_{p\in \tilde{J}(x,m) \atop p\equiv b' \overline{m'} \bmod q'} e(\beta p m)
=\frac{1}{\varphi(q')}\sum_{p\in \tilde{J}(x,m)} e(\beta p m)+E_2,
$$
where
$$
E_2 \ll \frac{x}{m} e^{-c_1 \sqrt{\log x/m}} +|\beta| x \int_{\tilde{J}(x,m)}e^{-c_1 \sqrt{\log \rho}} d \rho =: E_3+E_4,
$$
say.
Since $m\le x \exp(-g(x)^{2/3})$ if $\tilde{J}(x,m)$ is non-empty, $E_3$ contributes
$$
\ll q \sum_{m\le x} \frac{x}{m} \exp(-c_1 \sqrt[3]{g(x)})  \ll x\exp(-2c \sqrt[3]{g(x)})=: E_5.
$$
The contribution from $E_4$ is
$$
\ll q |\beta| x \sum_{m\le x} \int_{\tilde{J}(x,m)} e^{-c_1 \sqrt{\log \rho}} d \rho 
\le  q |\beta| x \sum_{m\le x} \frac{x}{m}\exp(-c_1 \sqrt[3]{g(x)})  \ll E_5,
$$
since $|\beta|x \le \exp(c \sqrt[3]{g(x)})$ and $c<c_1/3$. 
Thus 
$$
f(a/q+\beta)=
  \sum_{d|q} \sum_{b=1 \atop (b,q)=d}^q e({\textstyle \frac{a}{q}}b)
\sum_{m\in \mathcal{B}(x) \atop (m,q)= d}
 \frac{1}{\varphi(q/d)}\sum_{p\in \tilde{J}(x,m) } e(\beta p m) + O(E_5).
$$
The two inner sums are now independent of $b$. Thus we can evaluate the sum over $b$
with Lemma \ref{lembsum}. 
We combine the two inner sums, writing $n=mp$, and put back the contribution 
from $P(n)\le \exp(g(x)^{2/3})$, with an error $\ll  E_5$.
Thus
$$
f(a/q+\beta) = \sum_{d|q}  \frac{\mu(q/d)}{\varphi(q/d)} \sum_{n \in \mathcal{B}(x) \atop (n,q)= d} e(\beta n) + O( E_5),
$$
which is \eqref{eq1thm2}. Since $|\mu(q/d)|\le 1$ and $|e(\beta n)|\le 1$, the modulus of the main term is  
$$
\le   \sum_{d|q}  \frac{1}{\varphi(q/d)} \sum_{n \in \mathcal{B}(x) \atop (n,q)=d} 1
\le \frac{1}{\varphi(q)}  \sum_{d|q}  d B_d(x).
$$
With the assumption $B_d(x) \ll d^{-\delta} B(x)$, where $0<\varepsilon < \delta \le 1$, this is
$$
\ll \frac{ B(x)}{\varphi(q)} \sum_{d|q}  d^{1-\delta}
\ll  \frac{ B(x)}{\varphi(q)} q^{1-\delta}  \sum_{d|q}1\ll \frac{ B(x)}{q^{\delta-\varepsilon}} .
$$

\section{Lemmas for the minor arcs}

\begin{lemma}\label{lemminarc}
Assume $|a_n|, |b_n| \le 1$ for $n\ge 1$ and $h,a,q \in \mathbb{N}$. For $|\alpha - a/q| \le 1/q^2$ with $(a,q)=1$, we have
$$\mathop{\sum_{n > N } \sum_{m> M}}_{mn \le x}  a_n b_m e(\alpha h n m) \ll 
\left( \frac{hx}{M}+\frac{x}{N}+\frac{hx}{q}+q\right)^{\frac{1}{2}} (hx)^{\frac{1}{2}} (\log 2hx)^2.
$$
\end{lemma}

\begin{proof}
If $h=1$, this is \cite[Lemma 13.8]{IK}. The general case follows from this case with $n'=hn$, $x'=hx$, $N'=hN$, $a'_{n'}=a_{n'/h}$
if $h|n'$ and $a'_{n'}=0$ otherwise.
\end{proof}

\begin{lemma}\label{lemprimes}
Let $h,m,a,q \in \mathbb{N}$. For $|\alpha - a/q| \le 1/q^2$ with $(a,q)=1$, we have
$$
\sum_{p\le x/m} e(\alpha hmp) 
\ll (\log 2hx)^7 \left(hx q^{-\frac{1}{2}} + (h^2 mx)^{\frac{4}{5}} + (hxq)^{\frac{1}{2}} \right).
$$
\end{lemma}
\begin{proof}
This follows from \cite[Thm. 1]{Vau}, with the substitutions $H=hm$, $N=x/m$, $D\le H=hm$. 
The factor $\log p$ in that result can be removed with partial summation.
\end{proof}

\section{Proof of Theorem \ref{thm3}}

If $n \in \mathcal{B}$, $n>1$, we write $n=mp$ where $p=P(n)$ and $m\in \mathcal{B}$. 
We have
\begin{equation*}
\begin{split}
\sum_{n\in \mathcal{B}(x)} e(\alpha h n) &
=e(\alpha h) +  \sum_{m\in \mathcal{B}(x) } \sum_{P(m)\le p \le x/m \atop p \in I(m)}e(\alpha h m p ).
\end{split}
\end{equation*}
We write $L:=\log x$. 
Lemma \ref{lemPsi} shows that when $R=\exp(\kappa \sqrt{\log x \log\log x})$, $\kappa<1/\sqrt{6}$, then 
$ \Psi(x, R^3 L^{6}) \ll x/R$. This upper bound must also hold for any smaller value of $R$. 
Thus the contribution from $p\le R^3 L^{6}$ is $\ll x/R$. 

By Lemma \ref{lemprimes}, the contribution from $m\le R^4 L^{6}$ is 
\begin{equation*}
\begin{split}
\le & \sum_{m\le R^4 L^{6}} \Bigl|\sum_{P(m)\le p\le x/m \atop p \in I(m)} e(\alpha hmp)\Bigr| \\
\ll &   R^4 L^{6} L^7 \left(hx q^{-\frac{1}{2}} + x^{\frac{4}{5}+\varepsilon} + (hxq)^\frac{1}{2}\right)
\ll  \frac{x}{R},
\end{split}
\end{equation*}
provided $h\le R$ and $R^{12} L^{26} < q \le \frac{x}{R^{11}L^{26}}$. 

Let $M =R^4 L^{6}$, $N =R^3 L^{6} $, and $a_m$ (resp. $b_n$) be the characteristic functions of $\mathcal{B}$ (resp. of the primes). 
It remains to estimate
$$
S:=   \mathop{\sum_{ m> M } \sum_{n> N, \, n \in \tilde{I}(m)}}_{mn \le x} a_m b_n e(\alpha h m n ) ,
$$
where $\tilde{I}(m)=I(m) \cap [P(m),\infty)$. 
Applying a strategy called \emph{cosmetic surgery} in \cite[Sec. 3.2]{Har}, we can remove the 
condition $n\in \tilde{I}(m)$ at the expense of 
a factor $\ll \log x$. This leads to 
$$
S \ll (\log x)  \Biggl| \mathop{\sum_{ m> M } \sum_{n> N}}_{mn \le x} \tilde{a}_m \tilde{b}_n e(\alpha h m n ) \Biggr| +O(1),
$$
for suitable $\tilde{a}_n$, $\tilde{b}_n$ with $|\tilde{a}_m|, |\tilde{b}_n| \le 1$. 
An application of Lemma \ref{lemminarc} with $R^4 L^{6}<q\le \frac{x}{R^3 L^{6}}$ and $h\le R$ yields
$S \ll x/R$.

\section{Proof of Theorem \ref{thm4}}\label{secdis}

For the following upper bound see \cite[Thm. 2.5 of Ch. 2]{KN}.
\begin{lemma}[Erd\H{o}s-Tur\'an]\label{lemET}
With the notation of Theorem \ref{thm4} and $m \in \mathbb{N}$ we have
$$
\sup_{0\le y \le 1}|B(x,y,\alpha)-yB(x)| \le \frac{6 B(x)}{m+1} + \frac{4}{\pi} \sum_{h=1}^m \left(\frac{1}{h}-\frac{1}{m+1}\right)
 \left| \sum_{n\in \mathcal{B}(x)} e(\alpha h n)\right|.
$$
\end{lemma}

\begin{proof}[Proof of Theorem \ref{thm4}]
Let $\kappa<1/\sqrt{6}$ and pick $\eta>0$ such that $\kappa+\eta < 1/\sqrt{6}$. 
Let $R=\exp((\kappa+\eta)\sqrt{\log x \log\log x})$. 
Since $\alpha $ has finite irrationality measure $\mu$, the continued fraction convergents $a_j/q_j$ of $\alpha$
satisfy $q_{j+1} \ll_\alpha q_j^\mu$.  It follows that for all $x\ge x_0(\alpha)$, there is a $q_j$ satisfying 
the condition 
$Q< q_j \le  \frac{x}{Q}$ 
 where $Q:=R^{12}(\log x)^{26}$. Theorem \ref{thm3} shows that for $h\le R$ and $|\alpha -a_j/q_j|\le 1/q_j^2$, 
$$
 \sum_{n\in \mathcal{B}(x)} e(\alpha h n) \ll x/R.
$$
The result now follows from Lemma \ref{lemET} with $m=\lfloor R \rfloor$. 
\end{proof}


\begin{thebibliography}{99}

\bibitem{deBru}
N. G. de Bruijn, 
On the number of positive integers $\le x$ and free prime factors $>y$. II,
\textit{Indag. Math.} \textbf{28} (1966) 239--247.

\bibitem{Har}
G. Harman, Prime-detecting sieves.
London Mathematical Society Monographs Series, 33. Princeton University Press, Princeton, NJ, 2007.

\bibitem{IK}
H. Iwaniec and E. Kowalski, 
Analytic number theory.
American Mathematical Society Colloquium Publications, 53. American Mathematical Society, Providence, RI, 2004.

\bibitem{KN}
L. Kuipers and H. Niederreiter,
Uniform distribution of sequences,
Pure and Applied Mathematics. Wiley-Interscience, New York-London-Sydney, 1974.


\bibitem{MV}
H. L. Montgomery and R. C. Vaughan, 
Multiplicative Number Theory I : Classical Theory,
Cambridge Studies in Advanced Mathematics 97,
Camebridge University Press, 2006.  

\bibitem{oeis}
OEIS Foundation Inc. (2023), The On-Line Encyclopedia of Integer Sequences, Published electronically at \url{http://oeis.org}

\bibitem{Poly}
D. H. J. Polymath, 
New equidistribution estimates of Zhang type,
\textit{Algebra Number Theory} \textbf{8} (2014), no. 9, 2067--2199.

\bibitem{PW}
C. Pomerance and A. Weingartner,
On primes and practical numbers,
\textit{Ramanujan J.} \textbf{57} (2022), no. 3, 981--1000.

\bibitem{ST}
A. Stef and G. Tenenbaum, 
Entiers lexicographiques,
\textit{Ramanujan J.} \textbf{2} (1998), no. 1-2, 167--184.

\bibitem{Vau}
R. C. Vaughan,
On the distribution of $\alpha p$ modulo 1,
\textit{Mathematika} \textbf{24} (1977), no. 2, 135--141.

\bibitem{Vino}
I. M. Vinogradov, 
The method of trigonometrical sums in the theory of numbers. 
Translated from Russian, revised and annotated by K. F. Roth and Anne Davenport,
Interscience Publishers, London and New York, 1954.

\bibitem{SPA}
A. Weingartner, A sieve problem and its application, 
\textit{Mathematika} \textbf{63} (2017), no. 1, 213--229. 

\bibitem{SW}
A. Weingartner, 
An extension of the Siegel-Walfisz theorem,
\textit{Proc. Amer. Math. Soc.} \textbf{149} (2021), no. 11, 4699--4708.

\bibitem{Tau}
A. Weingartner,
The mean number of divisors for rough, dense and practical numbers, to appear in \textit{Int. J. Number Theory}, arXiv:2104.07137.

\bibitem{Weis}
E. W. Weisstein, Irrationality Measure, From MathWorld--A Wolfram Web Resource; \url{https://mathworld.wolfram.com/IrrationalityMeasure.html}



\end{thebibliography}
\end{document}